\documentclass[draft,reqno]{amsproc}
\usepackage{amsmath}
\usepackage{amsfonts}
\usepackage{mathrsfs}
\usepackage[dvips]{graphicx}
\usepackage{amssymb}
\usepackage{amsbsy}
\usepackage{amsthm}
\usepackage{graphics}
\usepackage{color}
\usepackage{textcomp}

\makeatletter
\@namedef{subjclassname@2010}{%
\textup{2010} Mathematics Subject Classification}
\makeatother

\DeclareMathOperator{\slim}{(\mbox{\sc sot})lim}

\numberwithin{equation}{section}

\newtheorem{theorem}{Theorem}[section]
\newtheorem{corollary}[theorem]{Corollary}
\newtheorem{prop}[theorem]{Proposition}

\newtheorem{prob}[theorem]{Problem}
\newtheorem{pyta}[theorem]{Question}
\newtheorem{lemma}[theorem]{Lemma}

\theoremstyle{remark}
\newtheorem{rem}[theorem]{Remark}

\newcommand*{\hh}{\mathcal{H}}

\newcommand*{\ascr}{\mathscr{A}}
\newcommand*{\bscr}{\mathscr{B}}

\newcommand*{\natu}{\mathbb{N}}
\newcommand*{\mfr}{\mathfrak{M}}

\newcommand*{\nul}{\mathcal{N}}
\newcommand*{\comp}{\mathbb{C}}
\newcommand*{\borel}{\mathfrak{B}}

\newcommand*{\cbb}{\comp}

\newcommand*{\D}{\mathrm{d\hspace{.1ex}}}

\theoremstyle{definition}
\newtheorem{ex}[theorem]{Example}

\newcommand*{\Le}{\leqslant}
\newcommand*{\Ge}{\geqslant}
\newcommand*{\ran}{\mathcal{R}}

\newcommand*{\ogr}{\boldsymbol{B}}
\newcommand*{\kk}{\mathcal{K}}
\newcommand*{\rbb}{\mathbb{R}}
\newcommand*{\zbb}{\mathbb{Z}}
   
\begin{document}

   \title[{Two-moment characterization of spectral
measures}] {Two-moment characterization of
spectral measures \\ on the real line}

   \author[P. Pietrzycki and J. Stochel]{Pawe{\l} Pietrzycki and  Jan Stochel}



   \address{Wydzia{\l} Matematyki i Informatyki, Uniwersytet
Jagiello\'{n}ski, ul. {\L}ojasiewicza 6, PL-30348
Krak\'{o}w}

   \email{pawel.pietrzycki@im.uj.edu.pl}

   \address{Wydzia{\l} Matematyki i Informatyki, Uniwersytet
Jagiello\'{n}ski, ul. {\L}ojasiewicza 6, PL-30348
Krak\'{o}w}

   \email{jan.stochel@im.uj.edu.pl}

   \begin{abstract}
In \cite{KLY06JMP}, Kiukas, Lahti and Ylinen
asked the following general question. {\em When
is a positive operator measure projection
valued?} A version of this question formulated
in terms of operator moments was posed in
\cite{P-S}. {\em Let $T$ be a selfadjoint
operator and $F$ be a Borel semispectral measure
on the real line with compact support. For which
positive integers $p< q$ do the equalities $T^k
=\int_{\rbb} x^k F(\D x)$, $k=p, q$, imply that
$F$ is a spectral measure?} In the present
paper, we completely solve the second problem.
The answer is affirmative if $p$ is odd and $q$
is even, and negative otherwise. The case
$(p,q)=(1,2)$ closely related to intrinsic noise
operator was solved by several authors including
Kruszy\'{n}ski and de Muynck as well as Kiukas,
Lahti and Ylinen. The counterpart of the second
problem concerning the multiplicativity of
unital positive linear maps on $C^*$-algebras is
also solved.
   \end{abstract}

   \maketitle

   \section{Introduction}
One of the most important concepts in
mathematics and physics is the notion of a
normalized positive operator valued measure also
known as a probability operator valued measure
or a generalized observable, or else
semispectral measure. This concept was
introduced in the 1940s by Naimark (see
\cite{Nai40,Nai40c2,Nai43}). Positive operator
valued measures play a significant role in
operator theory \cite{Ber66,Ak-Gl93,Ju-St08,P-S}
and are a standard tool in quantum information
theory and quantum optics
\cite{Brand99,Ve07,Ho11,B-L-P-Y16}. Recall that
a map $F\colon \ascr \to \ogr(\hh)$ defined on a
$\sigma$-algebra $\ascr$ of subsets of a set $X$
is said to be:
   \begin{itemize}
   \item a \textit{positive operator valued measure}
({\em POV measure}) if $\langle F (\cdot)h,
h\rangle$ is a positive measure for every $h \in
\hh$,
   \item a \textit{semispectral measure} if $F$ is a POV
measure such that $F(X) = I$,
   \item a \textit{spectral measure} if $F$ is a
semispectral measure such that $F(\varDelta)$ is
an orthogonal projection for every $\varDelta
\in \ascr$,
   \end{itemize}
where $\ogr(\hh)$ is the collection of all
bounded linear operators on a Hilbert space
$\hh$ and $I$ is the identity operator on $\hh$.
The celebrated Naimark's dilation theorem (see
\cite{Nai43} and \cite[Theorem~6.4]{Ml78})
states that a POV measure $F\colon \ascr \to
\ogr(\hh)$ can always be represented as the
$R$-compression $R^*E(\cdot)R$ of a spectral
measure $E\colon \ascr \to \ogr(\kk)$, where
$\kk$ is a Hilbert space and $R$ is a bounded
linear operator from $\hh$ to $\kk$. By
\cite[p.\ 14]{Ml78}, $\kk$ can be made minimal
in the sense that $\kk = \bigvee
\{E(\varDelta)R(\hh) \colon \varDelta \in
\ascr\}$. If $F$ is semispectral, then $\hh$ is
a subspace of $\kk$ and $R$ is the (isometric)
embedding of $\hh$ into $\kk$, and so the
minimality condition takes the form $\kk =
\bigvee\{E(\varDelta)\hh \colon \varDelta \in
\ascr\}$.

It turns out that, from a mathematical and
physical point of view, it is important to
investigate the relationship between
semispectral and spectral measures. In the
classical von Neumann description of quantum
mechanics selfadjoint operators or,
equivalently, Borel spectral measures on the
real line represent observables. This approach
is insufficient in describing many natural
properties of measurements, such as measurement
inaccuracy. Therefore, in standard modern
quantum theory, the generalization to
semispectral measures is widely used. In
particular, this is the case in quantum
information theory and in quantum optics (to
represent measurement statistics). Among the
papers undertaking this line of research, the
following are noteworthy
\cite{Nai40c2,hol72,K-dM87,Cat00,Ben03,Ben06,KLY06JMP,Jen07,Ben10,Ben16}.

By a {\em Borel POV measure on $\rbb$} we mean a
POV measure $F\colon \borel(\rbb) \to
\ogr(\hh)$, where $\borel(\rbb)$ stands for the
$\sigma$-algebra of all Borel subsets of the
real line $\rbb$ (below, the algebra $\ogr(\hh)$
will not be explicitly mentioned unless
necessary). For an integer $n\Ge 1$ and a Borel
POV measure $F$ on $\rbb$ with compact
support\footnote{For more information on closed
supports of spectral and semispectral measures,
see \cite[p.\ 69]{Sch12} and \cite[p.\
1799]{Ju-St08}.}, the integral
   \begin{equation*}
\int_{\rbb} x^{n} F(\D x)
   \end{equation*}
is a (bounded) self-adjoint operator, which is
called the {\em $n$th operator moment} of $F$. A
straightforward application of the Weierstrass
approximation theorem shows that a Borel POV
measure on $\rbb$ with compact support is
uniquely determined by its operator moments. One
of the features of a Borel spectral measures on
$\rbb$ is the multiplicativity of the
corresponding Stone-von Neumann functional
calculus. In particular, if $E$ is a Borel
spectral measure on $\rbb$ with compact support,
then the following identities hold\footnote{The
identity \eqref{polsp} holds even if the closed
support of $E$ is not compact. Since we only
deal with bounded operators in this paper, the
POV measures considered have compact supports
(see Section~\ref{Sec.7} for more explanation).}
  \begin{align} \label{polsp}
\Big(\int_{\rbb} x E(\D x)\Big)^{n} =
\int_{\rbb} x^{n} E(\D x), \quad n= 1, 2,
\ldots.
   \end{align}
Hence, all operator moments of $E$ are
determined by the first one, and according to
the spectral theorem there is a one-to-one
correspondence between Borel spectral measures
on $\rbb$ and their first operator moments. This
is no longer true for general Borel semispectral
measures on $\rbb$. It turns out, however, that
the single equality in \eqref{polsp} with $n=2$
guarantees spectrality.
   \begin{theorem}[{\cite[Proposition~1]{K-dM87},
\cite[Theorem~5]{KLY06JMP} and
\cite[Remark~5.3]{P-S}}] \label{kukuthe} A~Borel
semispectral measure $F$ on $\rbb$ with compact
support\/\footnote{The first two references
contain versions of this result for semispectral
measures with non-compact supports.} is spectral
if and only~if
   \begin{align*}  \Big(\int_{\rbb} x
F(\D x)\Big)^{2} = \int_{\rbb} x^{2} F(\D x).
   \end{align*}
   \end{theorem}
It is worth mentioning that if $F$ is a Borel
semispectral measure on $\rbb$ with compact
support, then the operator $\mathrm{Var}(F)$,
called {\em intrinsic noise operator} (see
\cite[p.\ ~177]{B-L-P-Y16}), defined by
  \begin{align} \label{noise}
\mathrm{Var}(F)=\int_{\rbb} x^{2} F(\D
x)-\Big(\int_{\rbb} x F(\D x)\Big)^{2}
  \end{align}
is always positive (see Corollary~\ref{istris};
this can also be deduced from the Kadison
inequality \eqref{kadison}). Thus, according to
Theorem~\ref{kukuthe}, equality holds in
$\mathrm{Var}(F) \Ge 0$ only for spectral
measures.

In this connection, it is worth emphasizing that
Theorem~\ref{kukuthe} was developed for the
needs of quantum physics. Namely, the main
purpose of the quantization proposed in
\cite{KLY06JMAA,KLY06JMP} was to construct
observables that are not spectral measures, and
this was done by using the operator moments of
these observables. To achieve this goal, it was
important to be able to use these moments to
determine whether a given observable is or is
not a spectral measure. This led Kiukas, Lahti
and Ylinen to the following question (see
\cite[Sec.\ ~VI]{KLY06JMP}, see also \cite[Sec.\
~5]{Ki07}):
   \begin{pyta} \label{putan}
When is a positive operator measure projection
valued\/{\em ?}
   \end{pyta}

In a recent paper \cite{P-S}, we gave a solution
to {\cite[Problem~1.1]{Curto20}} concerning
subnormal square roots of quasinormal operators.
In fact, the paper \cite{P-S} provides two
solutions to this problem that use two different
approaches. The first one appeals to the theory
of operator monotone functions, in particular
Hansen's inequality. The second is based on the
technique that utilizes operator moments of
semispectral measures. A detailed analysis of
both solutions led us to a new criterion for the
spectrality of a Borel semispectral measure on
$\rbb$ compactly supported in $[0,\infty)$,
written in terms of its two operator moments.
This criterion was used to solve a
generalization of {\cite[Problem~1.1]{Curto20}}
(see \cite[Theorem 4.1]{P-S}).
   \begin{theorem}[{\cite[Theorem 4.2]{P-S}, \cite{P-S-cor}}]
\label{maintw4} Let $T\in \ogr(\hh)$ be a
positive operator and $\alpha,\beta$ be two
distinct positive real numbers. Assume that
$F\colon \borel(\rbb) \to \ogr(\hh)$ is a
semispectral measure compactly supported in
$[0,\infty)$. Then the following conditions
are~equivalent{\em :}
   \begin{enumerate}
   \item[(i)] $F$ is the spectral
measure of $T$,
   \item[(ii)] $T^n = \int_{[0,\infty)} x^n F(\D
x)$ for all integers $n\Ge 0$,
   \item[(iii)] $T^r = \int_{[0,\infty)} x^r F(\D
x)$ for all $r \in [0,\infty)$,
   \item[(iv)] $T^r =\int_{[0,\infty)}
x^r F(\D x)$ for $r= \alpha, \beta$.
   \end{enumerate}
   \end{theorem}
As shown in the proof of
\cite[Theorem~4.2]{P-S}, the implication
(iv)$\Rightarrow$(i) is equivalent to the fact
that a semispectral measure $F\colon
\borel(\rbb) \to \ogr(\hh)$ compactly supported
in $[0,\infty)$ for which there exists $s\in
(0,\infty) \backslash \{1\}$ such that
   \begin{align*}
\Big(\int_{[0,\infty)} x F(\D x)\Big)^{s} =
\int_{[0,\infty)} x^{s} F(\D x)
   \end{align*}
is spectral (see \cite[Lemma~4.3]{P-S}).

In view of Question~\ref{putan} and
Theorems~\ref{kukuthe} and \ref{maintw4}, it
seems natural to pose the following general
problem in which $\varXi$ is a fixed nonempty
set of positive integers (if $\varXi$ is finite,
then we always order its elements in a
non-decreasing manner). Problem~\ref{momentprob}
below can be regarded as a generalization of
\cite[Problem~5.2]{P-S} which deals with
two-element sets $\varXi$.
   \begin{prob} \label{momentprob}
Let $T\in \ogr(\hh)$ be a selfadjoint operator
and $F\colon \borel(\rbb) \to \ogr(\hh)$ be a
semispectral measure with compact support. Does
the system of equations
   \begin{equation}\label{momprob}
T^k =\int_{\rbb} x^k F(\D x),\quad k \in \varXi,
   \end{equation}
imply that $F$ is spectral\/{\em ?}
   \end{prob}
This problem can be rephrased equivalently in
terms of dilation theory as follows (use
Naimark's dilation theorem and
Lemma~\ref{fuglemma}):

   \begin{prob}\label{dilationprob}
Let $T\in \ogr(\hh)$ be a selfadjoint operator,
$F\colon \borel(\rbb) \to \ogr(\hh)$ be a
semispectral measure with compact support,
$E\colon \borel(\rbb) \to \ogr(\kk)$ be a
minimal spectral dilation of $F$ $($i.e., $E$ is
a spectral measure satisfying \eqref{fpeh} and
\eqref{minimity}$)$ and $S$ be the first
operator moment of $E$ $($i.e., $S=\int_{\rbb} x
E(\D x)$$)$. Does the system of~equations
   \begin{equation*}
T^k = P S^k|_{\hh}, \quad k\in \varXi,
   \end{equation*}
imply that $P$ and $S$ commutes\/{\em ?}
   \end{prob}
It turns out that Problem~\ref{momentprob} is
closely related to the question of
multiplicativity of unital positive linear maps
on $C^*$-algebras (see Remark~\ref{takisubi}).
In fact, the two problems are logically
equivalent regardless of the cardinality of the
set $\varXi$ (see Remark~\ref{rown}). The
$C^*$-algebra counterpart of
Problem~\ref{momentprob} takes the following
form.
   \begin{prob} \label{algebraprob}
Let $\mathcal{A}$ and $\mathcal{B}$ be unital
$C^*$-algebras, $\varPhi \colon \mathcal{A}
\rightarrow \mathcal{B}$ be a unital positive
linear map and $a\in \mathcal{A}$ and $b\in
\mathcal{B}$ be selfadjoint. Does the system of
equations
   \begin{equation*}
b^k = \varPhi(a^k), \quad k\in \varXi,
   \end{equation*}
imply that $\varPhi$ restricted to the unital
subalgebra generated by $\{a\}$ is
multiplicative\/{\em ?}
   \end{prob}
The correspondence between
Problems~\ref{momentprob} and \ref{dilationprob}
allows us to use the theory of operator monotone
functions and related operator inequalities to
prove the main results of this paper, which
provide complete solutions to
Problems~\ref{momentprob}, \ref{dilationprob}
and \ref{algebraprob} for two-element sets
$\varXi$. We begin with the affirmative
solutions.
   \begin{theorem} \label{main1}
Let $T \in \ogr(\hh)$ be a selfadjoint operator,
$F\colon \borel(\rbb) \to \ogr(\hh)$ be a
semispectral measure with compact support and
$p,q$ be positive integers such that $p<q$, $p$
is odd and $q$ is even. Then the following
conditions are equivalent{\em :}
   \begin{enumerate}
   \item[(i)] $F$ is the spectral measure of $T$,
   \item[(ii)] $T^k =\int_{\rbb} x^k F(\D x)$
for all integers $k\Ge 0$,
   \item[(iii)]  $T^k =\int_{\rbb} x^k F(\D x)$ for
$k=p, q$.
   \end{enumerate}
   \end{theorem}
The affirmative solution to
Problem~\ref{algebraprob} takes the form.
   \begin{theorem} \label{poturs}
Let $\mathcal{A}$ and $\mathcal{B}$ be unital
$C^*$-algebras, $\varPhi \colon \mathcal{A}
\rightarrow \mathcal{B}$ be a unital positive
linear map, $a\in \mathcal{A}$ be selfadjoint
and $p,q$ be positive integers such that $p<q$,
$p$ is odd and $q$ is even. Then the following
conditions are equivalent{\em :}
   \begin{enumerate}
   \item[(i)] $\varPhi$ restricted to
the unital subalgebra generated by $\{a\}$ is
multiplicative,
   \item[(ii)] there exists a selfadjoint
element $b\in \mathcal{B}$ such that $b^k =
\varPhi(a^k)$ for $k=p,q$.
   \end{enumerate}
Moreover, if {\em (ii)} holds, then
$b=\varPhi(a)$.
   \end{theorem}
In the complementary result, we show that the
set (with $\natu=\{1,2,3, \ldots\}$)
   \begin{align} \label{omigra}
\varOmega:=\{(p,q)\in \natu^2\colon p<q, \, p
\text{ odd and } q \text{ even}\}
   \end{align}
is the largest possible subset of $\{(p,q)\in
\natu^2\colon p \Le q\}$ for which
Problem~\ref{momentprob} has an affirmative
solution for $\varXi=\{p,q\}$. Surprisingly,
suitable counterexamples can be constructed even
when the underlying Hilbert space $\hh$ is
one-dimensional (see Theorem~\ref{main2} for
more details).
   \begin{theorem} \label{main2-w0}
Let $(p,q)\in \natu^2\backslash\varOmega$ be
such that $p\Le q$. Then there exist a Hilbert
space $\hh$, a selfadjoint operator $T\in
\ogr(\hh)$ and a semispectral measure $F\colon
\borel(\rbb) \to \ogr(\hh)$ with compact support
which is not spectral and such that
   \begin{equation*}
T^k=\int_{\rbb} x^k F(\D x), \quad k=p,q.
   \end{equation*}
   \end{theorem}
The proofs of Theorems~\ref{main1}, \ref{poturs}
and \ref{main2-w0} will be given in
Sections~\ref{Sec.3}, \ref{Sec.4} and
\ref{Sec.5}, respectively. In
Section~\ref{Sec.2} we provide the basic facts
on operator monotone functions and the related
operator inequalities needed in this paper.
Section~\ref{Sec.6} contains additional
counterexamples (including the case of infinite
dimensional spaces) related to the Fibonacci
sequence. Finally, in Section~\ref{Sec.7} we
discuss the possibility of adapting the
two-moment characterizations of spectral
measures given in Theorems~\ref{maintw4} and
\ref{main1} to the case of semispectral measures
whose closed supports are not compact.
   \section{\label{Sec.2}Prerequisites}
In this paper, we use the following notation.
The fields of real and complex numbers are
denoted by $\rbb$ and $\mathbb{C}$,
respectively. The symbols $\zbb_{+}$,
$\mathbb{N}$ and $\rbb_+$ stand for the sets of
nonnegative integers, positive integers and
nonnegative real numbers, respectively. We write
$\borel(X)$ for the $\sigma$-algebra of all
Borel subsets of a topological Hausdorff space
$X$. The $C^*$-algebra of all continuous complex
functions on a compact Hausdorff space $K$
equipped with supremum norm is denoted by
$C(K)$. For $\lambda \in \rbb$,
$\delta_{\lambda}$ stands for the Borel
probability measure on $\rbb$ concentrated on
$\{\lambda\}$.

Let $\hh$ and $\kk$ be (complex) Hilbert spaces.
Denote by $\ogr(\hh, \kk)$ the Banach space of
all bounded linear operators from $\hh$ to
$\kk$. If $A\in \ogr(\hh,\kk)$, then $A^*$,
$\nul(A)$ and $\ran(A)$ stand for the adjoint,
the kernel and the range of $A$, respectively.
It is well known that $\ogr(\hh):=\ogr(\hh,\hh)$
is a $C^*$-algebra with unit $I$, where
$I=I_{\hh}$ denotes the identity operator on
$\hh$. We say that $A\in \ogr(\hh)$ is
\textit{selfadjoint} if $A=A^*$,
\textit{positive} if $\langle Ah,h\rangle \Ge 0$
for all $h\in \hh$ and an \textit{orthogonal
projection} if $A=A^*$ and~$A=A^2$.

Let $\ascr$ be a $\sigma$-algebra of subsets of
a set $X$ and let $F\colon \ascr \to \ogr(\hh)$
be a semispectral measure. Denote by $L^1(F)$
the vector space of all $\ascr$-measurable
functions $f\colon X \to \cbb$ such that
$\int_{X} |f(x)| \langle F(\D x)h, h\rangle <
\infty$ for all $h\in \hh$. Then for every $f\in
L^1(F)$, there exists a unique operator $\int_X
f \D F \in \ogr(\hh)$ such that (see e.g.,
\cite[Appendix]{Sto92})
   \begin{align} \label{form-ua}
\Big\langle\int_X f \D F h, h\Big\rangle
= \int_X f(x) \langle F(\D x)h, h\rangle,
\quad h\in\hh.
   \end{align}
If $F$ is a spectral measure, then $\int_X f \D
F$ coincides with the usual spectral integral.
In particular, if $F$ is the spectral measure of
a selfadjoint operator $A\in \ogr(\hh)$, then we
write $f(A)=\int_{\mathbb{R}} f \D F$ for any
$F$-essentially bounded Borel function $f\colon
\rbb \to \rbb$; the map $f \mapsto f(A)$ is
called the Stone-von Neumann functional
calculus. For more information needed in this
article on spectral integrals, including the
spectral theorem for selfadjoint operators and
the Stone-von Neumann functional calculus, we
refer the reader to \cite{Rud73,Weid80,Sch12}.

Let $J\subseteq \rbb$ be an interval (which may
be open, half-open, or closed; bounded or
unbounded). A continuous function $f \colon J
\rightarrow \rbb$ is said to be \textit{operator
monotone} if $f(A)\Le f(B)$ for any two
selfadjoint operators $A,B\in\ogr(\hh)$ such
that $A\Le B$ and the spectra of $A$ and $B$ are
contained in $J$. In \cite{Lo34}, L\"owner
proved that a continuous function defined on an
open interval is operator monotone if and only
if it has an analytic continuation to the
complex upper half-plane which is a Pick
function (see also \cite{Dono74,Han13}).
Operator monotone functions have integral
representations with respect to suitable
positive Borel measures. In particular, a
continuous function \mbox{$f \colon (0, \infty)
\rightarrow \rbb$} is operator monotone if and
only if there exists a positive Borel measure
$\nu$ on $[0,\infty)$ such that $\int_0^\infty
\frac{1}{1+\lambda^2}\D \nu(\lambda)<\infty$ and
   \begin{equation*} \label{repbol}
f(t)=\alpha +\beta t +\int_0^\infty
\Big(\frac{\lambda}{1+\lambda^2} -
\frac{1}{t+\lambda}\Big) \D \nu(\lambda),
\quad t \in (0,\infty),
   \end{equation*}
where $\alpha\in \rbb$ and $\beta \in \rbb_+$
(see \cite[Theorem~5.2]{Han13} or \cite[p.\
~144]{Bha97}). The most important example of an
operator monotone function is $f\colon
[0,\infty)\ni t\rightarrow t^p\in\rbb$ for $p\in
(0,1)$. This function has the following integral
representation (see
\cite[Exercise~V.1.10(iii)]{Bha97} or
\cite[Exercise~V.4.20]{Bha97})
   \begin{equation} \label{Hync-Luv}
t^p=\frac{\sin p\pi}{\pi} \int_0^\infty
\frac{t\lambda^{p-1}}{t+\lambda} \D
\lambda, \quad t \in [0,\infty).
   \end{equation}
Operator monotone functions are related to the
Hansen inequality \cite{Han80}. In
\cite[Lemma~2.2]{uch93}, Uchiyama gave a
necessary and sufficient condition for equality
to hold in the Hansen inequality when the
external factor is a nontrivial orthogonal
projection (see the ``moreover'' part of
Theorem~\ref{hans} below; see also the paragraph
before \cite[Theorem 2.4]{P-S} showing why the
separability of $\hh$ can be dropped).
   \begin{theorem}[\cite{Han80,uch93}]
\label{hans} Let $A\in \ogr(\hh)$ be a
positive operator, $T \in \ogr(\hh)$ be
a contraction and $f\colon
[0,\infty)\rightarrow \rbb$ be a
continuous operator monotone function
such that $f(0)\Ge 0$. Then
   \begin{equation} \label{Han-inq}
T^*f(A)T \Le f(T^*AT).
   \end{equation}
Moreover, if $f$ is not an affine function and
$T$ is an orthogonal projection such that $T
\neq I$, then equality holds in \eqref{Han-inq}
if and only if $TA=AT$ and $f(0)=0$.
   \end{theorem}
The reader is referred to
\cite{Lo34,Dono74,Han80,Bha97,Han13,Sim19} for
the fundamentals of the theory of operator
monotone functions.

A linear map $\varPhi \colon \mathcal{A}
\rightarrow \mathcal{ B}$ between unital
$C^*$-algebras is said to be {\em positive} if
$\varPhi(a) \Ge 0$ for every $a\in \mathcal{A}$
such that $a \Ge 0$. The map $\varPhi$ is called
{\em unital} if it preserves the units. If
$\varPhi$ is positive and unital, then the
following inequality, called Kadison's
inequality (see \cite{Kad52}), holds:
   \begin{equation} \label{kadison}
\text{$\varPhi(a^2)\Ge \varPhi(a)^2$ for all $a
\in \mathcal{A}$ such that $a=a^*$.}
   \end{equation}
In this paper we will need the following
generalization of Kadison's inequality.
   \begin{theorem}[{\cite[Theorem 2]{L-R74}}] \label{L-R}
Let $R\in \ogr(\hh,\kk)$ and let
$\varPhi\colon\ogr(\kk)\to\ogr(\hh)$ be the
positive linear map defined by
   \begin{align*}
\varPhi(X)=R^*XR,\quad X\in \ogr(\kk).
   \end{align*}
Then for all $A,B\in\ogr(\kk)$, the net
$\{\varPhi(A^*B)(\varPhi(B^*B)+\varepsilon
I)^{-1}\varPhi(B^*A)\}_{\varepsilon>0}$ is
convergent in the strong operator topology as
$\varepsilon \downarrow 0$ and
   \begin{align*}
\varPhi(A^*A)\Ge {\slim}_{\varepsilon\downarrow
0} \varPhi(A^*B)(\varPhi(B^*B)+\varepsilon
I)^{-1}\varPhi(B^*A).
   \end{align*}
   \end{theorem}
   \section{\label{Sec.3}Proof of Theorem~\ref{main1}}
We begin with the following lemma which gives a
necessary and sufficient condition for equality
to hold in a Kadison type inequality (cf.\
\eqref{kadison}). Although this is a known fact
even for unbounded operators (see \cite[Lemmas~1
and 2]{K-dM87}), we will provide a brief
algebraic proof for the reader's convenience.
Note also that part (iii) of
Lemma~\ref{kadlemma} below is \cite[Lemma in
Sec.\ ~6]{Fug83}.
   \begin{lemma}
\label{kadlemma} Let $T\in \ogr(\hh)$ be a
selfadjoint operator and $P \in \ogr(\hh)$ be an
orthogonal projection. Then the following
statements are valid{\em :}
   \begin{enumerate}
   \item[(i)] $(PTP)^2 \Le PT^{2}P$,
   \item[(ii)] equality holds in {\em (i)} if and only if
$PT=TP$,
   \item[(iii)] if $T$ is an orthogonal projection, then $PTP$
is an orthogonal projection if and only if $PT=TP$.
   \end{enumerate}
   \end{lemma}
   \begin{proof} (i) This is a direct consequence of the
following algebraic identities:
   \begin{align} \notag
PT^{2}P-(PTP)^2 & = PT^{2}P- PTPTP
   \\ \notag
& = PT(I-P)TP
   \\ \label{kadP2}
& = (TP)^*(I-P)TP \Ge 0.
   \end{align}

(ii) It follows from \eqref{kadP2} that equality holds in (i)
if and only if
   \begin{equation*}
(TP)^*(I-P)TP=0,
   \end{equation*}
or equivalently if and only if
   \begin{align*}
\ran(TP)\subseteq \nul((I-P)^{\frac{1}{2}})=\nul
(I-P),
   \end{align*}
which in turn is equivalent to $(I-P)TP=0$. The last equality
holds if and only if $TP=PTP$, which by $(PTP)^*=PTP$ is
equivalent to $PT=TP$.

(iii) This is a direct consequence of (ii) because $PTP$ is
an orthogonal projection if and only if $(PTP)^2 = PT^2P$.
   \end{proof}
For our further considerations, the following
fact is fundamental. In particular, in view of
Naimark's dilation theorem (see Introduction),
it shows that Problems \ref{momentprob} and
\ref{dilationprob} are logically equivalent.
   \begin{lemma} \label{fuglemma}
Let $\hh$, $\kk$ be Hilbert spaces such that
$\hh\subseteq\kk$ and $P\in \ogr(\kk)$ be the
orthogonal projection of $\kk$ onto $\hh$.
Suppose that $F\colon \borel(\rbb) \to
\ogr(\hh)$ is a semispectral measure and
$E\colon \borel(\rbb) \to \ogr(\kk)$ is a
spectral measure such that
   \begin{equation} \label{fpeh}
F(\varDelta) = PE(\varDelta)|_\hh,\quad \varDelta \in
\borel(\rbb).
   \end{equation}
Set\/\footnote{Note that a priori the operator
$S$ may be unbounded (see
\cite[Theorem~5.9]{Sch12} for more details).}
$S:=\int_{\rbb} x E(\D x)$. Then the following
statements are valid{\em :}
   \begin{itemize}
   \item[(i)] $F$ is  spectral if and only if
$P$ commutes with~$E$ $($equivalently, $\hh$
reduces~$E$$)$,
   \item[(ii)] if $S\in
\ogr(\kk)$, then $F$ has compact support and
   \begin{align} \label{wrestie}
P S^k|_{\hh}=\int_{\rbb} x^k F(\D x), \quad k\in
\zbb_+,
   \end{align}
   \item[(iii)] if $F$ has compact support and
$\kk$ is minimal, that is,
   \begin{align} \label{minimity}
\kk = \bigvee\{E(\varDelta)\hh \colon \varDelta
\in \ascr\},
   \end{align}
then $E$ has compact support, $S\in \ogr(\kk)$
and $S=S^*$.
   \end{itemize}
   \end{lemma}
   \begin{proof}
(i) Set $\hat{F}(\varDelta) = F(\varDelta)\oplus
0$ for $\varDelta \in \borel(\rbb)$, where $0$
stands for the zero operator on $\kk \ominus
\hh$. Then, by \eqref{fpeh}, $\hat{F}(\varDelta)
= PE(\varDelta)P$. Hence, observing that
$F(\varDelta)$ is an orthogonal projection if
and only if $\hat{F}(\varDelta)$ is an
orthogonal projection and using
Lemma~\ref{kadlemma}(iii), we obtain (i).

(ii) It follows from \eqref{fpeh} that the
closed support of $F$ is contained in the closed
support of $E$. Since $E$ has compact support
(because $S\in \ogr(\kk)$, see
\cite[Theorem~5.9]{Sch12}), so does $F$.
Applying the Stone-von Neumann functional
calculus, we~get
   \allowdisplaybreaks
   \begin{align*}
\langle P S^k|_{\hh} h, h \rangle = \langle S^k
h, h \rangle & = \int_{\rbb} x^k \langle E(\D
x)h,h\rangle
   \\
& \hspace{-1ex}\overset{\eqref{fpeh}}=
\int_{\rbb} x^k \langle F(\D x)h,h\rangle
   \\
& \hspace{-1ex}\overset{\eqref{form-ua}}=
\Big\langle \int_{\rbb} x^k F(\D
x)h,h\Big\rangle, \quad h \in \hh, \, k \in
\zbb_+,
   \end{align*}
which implies \eqref{wrestie}.

(iii) By \eqref{fpeh} and \eqref{minimity}, the
closed supports of the POV measures $E$ and $F$
coincide (see the proofs of
\cite[Theorem~4.4]{Ja02} and
\cite[Proposition~4(iii)]{Ju-St08}). Hence, the
closed support of $E$ is compact. As a
consequence, the operator $\int_{\rbb} x E(\D
x)$ is bounded and selfadjoint (see
\cite[Theorem~5.9]{Sch12}). This completes the
proof.
   \end{proof}
We are now in a position to prove the main
result of this paper, which provides a
two-moment characterization of spectral
measures.
   \begin{proof}[Proof of Theorem~\ref{main1}]
(i)$\Rightarrow$(ii) This is immediate from the
Stone-von Neumann functional calculus.

(ii)$\Rightarrow$(iii) Obvious.

(iii)$\Rightarrow$(i) It follows from Naimark's
dilation theorem (see Introduction) that there
exist a Hilbert space $\kk$ containing $\hh$ and
a spectral measure $E\colon \borel(\rbb) \to
\ogr(\kk)$ such that \eqref{fpeh} and
\eqref{minimity} hold, where $P\in \ogr(\kk)$ is
the orthogonal projection of $\kk$ onto $\hh$.
By Lemma~\ref{fuglemma}, $E$ has compact
support, the operator $S:=\int_{\rbb} x E(\D x)$
is bounded and selfadjoint, and the following
equalities are satisfied:
   \begin{align} \label{kon-1}
T^k = P S^k|_{\hh}, \quad k=p,q.
   \end{align}
First, we prove that $F$ is a spectral measure.
In view of Lemma~\ref{fuglemma}(i), it suffices
to show that $P$ commutes with $E$. For this, we
consider two cases.

{\sc Case 1.} $p\Le\ \frac{q}{2}$.

Let $\hat{T} \in \ogr(\kk)$ be defined by
$\hat{T} = T \oplus 0$, where $0$ stands for the
zero operator on $\kk \ominus \hh$. Set
$q^\prime=\frac{q}{2}$. Using
Lemma~\ref{kadlemma}(i) and then applying
Theorem~\ref{hans} to the positive operator
$S^{2q^\prime}$ and the operator monotone
function $f(t)=t^{\frac{p}{q^\prime}}$ (see
\eqref{Hync-Luv}), we deduce~that
   \allowdisplaybreaks
   \begin{align*}
\hat{T}^{2p} = (\hat{T}^{p})^2 &
\overset{\eqref{kon-1}} = (PS^{p}P)^2
   \\
& \hspace{.9ex} \Le PS^{2p}P
   \\
& \hspace{.9ex} \Le
(PS^{2q^\prime}P)^\frac{p}{q^\prime} =
(PS^{q}P)^\frac{2p}{q} \overset{\eqref{kon-1}} =
(\hat{T}^q)^{\frac{2p}{q}} \overset{(*)}=
\hat{T}^{2p},
   \end{align*}
where $(*)$ can be inferred from the hypothesis
that $q$ is even. This implies that
   \begin{equation*}
(PS^{p}P)^2= PS^{2p}P.
   \end{equation*}
It follows from Lemma~\ref{kadlemma}(ii) that
   \begin{equation*}
PS^{p}= S^{p}P.
   \end{equation*}
Hence, by \cite[Theorem~5.1]{Sch12}, $P$ commutes with
$E_{p}$, the spectral measure of $S^{p}$. By
\cite[Theorem~6.6.4]{Bir-Sol87}, $E_{p}$ is of the form
   \begin{align} \label{nosczeg}
E_{p}(\varDelta)=E(\varphi_{p}^{-1}(\varDelta)),
\quad \varDelta \in \borel(\rbb),
   \end{align}
where $\varphi_p\colon \rbb \to \rbb$ is a
function given by
   \begin{align} \label{fiub}
\varphi_{p}(x)=x^{p}, \quad x\in \rbb.
   \end{align}
Since the map $\borel(\rbb) \ni \varDelta
\mapsto \varphi_p^{-1}(\varDelta) \in
\borel(\rbb)$ is surjective (because $p$ is
odd), we deduce from \eqref{nosczeg} that $P$
commutes with $E$.

{\sc Case 2.} $p > \frac{q}{2}$.

Suppose, to the contrary, that $P$ does not
commute with $E$. This implies that $P\neq
I_{\kk}$. Set $q^\prime=\frac{q}{2}$ and
$r=p-q^\prime$. Since $p<q$ and $q$ is even, we
see that $r, q^\prime \in \natu$ and $0 <
\frac{r}{q^\prime}<1$. By Theorem~\ref{hans}
applied to the positive operator
$S^{2q^{\prime}}$ and the operator monotone
function $f(t)=t^{\frac{r}{q^\prime}}$, we~get
   \begin{align*}
\hat{T}^{2r} =
(\hat{T}^{2q^\prime})^{\frac{r}{q^\prime}}
\overset{\eqref{kon-1}} =
(PS^{2q^\prime}P)^{\frac{r}{q^\prime}} \Ge
PS^{2r}P.
   \end{align*}
This implies that
   \begin{equation} \label{3gw}
T^{2r}=(PS^{2q^\prime}|_\hh)^{\frac{r}{q^\prime}}\Ge PS^{2r}|_\hh.
   \end{equation}
Let $\varPhi\colon \ogr(\kk)\to \ogr(\hh)$ be
the positive unital linear map defined by
   \begin{equation*}
\varPhi(X) = P X|_{\hh}, \quad X\in \ogr(\kk).
   \end{equation*}
Applying Theorem \ref{L-R} to $A=S^r$,
$B=S^{q^{\prime}}$ and $R\in \ogr(\hh,\kk)$
defined by $Rh=h$ for $h\in \hh$ leads to
   \begin{align}\label{5gw}
P S^{2r}|_{\hh}=\varPhi(S^{2r}) \Ge
{\slim}_{\varepsilon\downarrow 0}
\varPhi(S^{p})(\varPhi(S^{q})+\varepsilon
I)^{-1}\varPhi(S^{p}).
   \end{align}
Let $G\colon \borel(\rbb) \to \ogr(\hh)$ be the
spectral measure of $T$. Using the Stone-von
Neumann functional calculus, we obtain
   \begin{align} \label{kra-kru}
\varPhi(S^{p})(\varPhi(S^{q})+\varepsilon
I)^{-1}\varPhi(S^{p}) \overset{\eqref{kon-1}} =
T^p(T^{q}+\varepsilon)T^p=\int_\mathbb{R}
\frac{x^{2p}}{x^{q}+\varepsilon}G(\D x).
   \end{align}
Applying Lebesgue's monotone convergence theorem
and the hypothesis that $q$ is even and $2p-q\in
\natu$, we deduce that
   \allowdisplaybreaks
   \begin{align*}
\lim_{\varepsilon\downarrow 0}
\langle\varPhi(S^{p}) (\varPhi(S^{q}) +
\varepsilon I)^{-1} \varPhi(S^{p})h, h\rangle &
\overset{\eqref{kra-kru}} =
\lim_{\varepsilon\downarrow 0}
\Big\langle\int_\mathbb{R}
\frac{x^{2p}}{x^q+\varepsilon} G(\D x)h,
h\Big\rangle
   \\
& \hspace{1ex}= \lim_{\varepsilon\downarrow 0}
\int_\mathbb{R}
\frac{x^{2p}}{x^q+\varepsilon}\langle G(\D
x)h,h\rangle
   \\
& \hspace{1ex}=\int_\mathbb{R} {x^{2p-q}}\langle
G(\D x)h, h\rangle
   \\
& \hspace{1ex}= \langle T^{2p-q}h,h\rangle =
\langle T^{2r}h,h\rangle, \quad h \in \hh.
   \end{align*}
Therefore, the net
$\{\varPhi(S^{p})(\varPhi(S^{q})+\varepsilon
I)^{-1}\varPhi(S^{p})\}_{\varepsilon>0}$ is
convergent to $T^{2r}$ in the weak operator
topology. Combined with \eqref{5gw}, this
implies that
   \begin{align} \label{dyv-ci}
P S^{2r}|_{\hh} \Ge T^{2r}.
   \end{align}
Using \eqref{3gw} and \eqref{dyv-ci}, we get
   \begin{equation*}\label{6gw}
T^{2r} =
(PS^{2q^\prime}|_\hh)^{\frac{r}{q^\prime}} \Ge
PS^{2r}|_\hh\Ge T^{2r}.
   \end{equation*}
This yields
   \begin{equation*}
(PS^{2q^\prime}|_\hh)^{\frac{r}{q^\prime}}=
PS^{2r}|_\hh,
   \end{equation*}
or equivalently
   \begin{equation*}
(PS^{2q^\prime}P)^{\frac{r}{q^\prime}}=
PS^{2r}P,
   \end{equation*}
so equality holds in the Hansen inequality.
Thus, by the moreover part of
Theorem~\ref{hans}, $PS^{q}=S^{q}P$ (recall that
$q=2q^\prime$). Hence
   \begin{align*}
\hat{T}^{q n}
\overset{\eqref{kon-1}}=(PS^{q}P)^n = (PS^{q})^n
= PS^{q n}P,\quad n\in \natu.
   \end{align*}
Therefore, $T^{q n}=PS^{q n}|_\hh$ for all $n\in
\natu$. Take any $n_0\in \natu$ such that $p\Le
\frac{q n_0}{2}$. Then by \eqref{kon-1}, we have
   \begin{align*}
T^k = P S^k|_{\hh},\qquad k=p, \, q n_0.
   \end{align*}
Since $p\Le \frac{q n_0}{2}$, we can apply
Case~$1$ to the pair $(p,q n_0)$ in place of
$(p,q)$. We then obtain that $P$ commutes with
$E$, which is a contradiction.

Summarizing, we have proved that in both cases
$F$ is a spectral measure. Therefore, to
complete the proof it remains to show that $F$
is the spectral measure of $T$. Since $T^p =
\int_{\rbb} x^p F(\D x)$ (by (iii)) and $T^p =
\int_{\rbb} x^p G(\D x)$ (by Stone-von Neumann
functional calculus), an application of
\cite[Theorem~6.6.4]{Bir-Sol87} shows that
$F\circ \varphi_{p}^{-1}$ and $G\circ
\varphi_{p}^{-1}$ are spectral measures of
$T^p$, where $G$ is the spectral measure of $T$,
$\varphi_{p}$ is as in \eqref{fiub} and
   \begin{align*}
\text{$(F\circ \varphi_{p}^{-1})(\varDelta) =
F(\varphi_{p}^{-1}(\varDelta))$ and $(G\circ
\varphi_{p}^{-1})(\varDelta) =
G(\varphi_{p}^{-1}(\varDelta))$ for $\varDelta
\in \borel(\rbb)$.}
   \end{align*}
By the uniqueness in
\cite[Theorem~6.1.1]{Bir-Sol87}, $F\circ
\varphi_p^{-1} = G\circ \varphi_p^{-1}$. Since
the map $\borel(\rbb) \ni \varDelta \mapsto
\varphi_p^{-1}(\varDelta) \in \borel(\rbb)$ is
surjective (because $p$ is odd), we deduce that
$F=G$, so $F$ is the spectral measure of $T$.
This completes the proof.
   \end{proof}
   We conclude this section by providing some
inequalities for moments of a semispectral
measure on the real line. Though it is a
well-known fact (see \cite{Bi94} and
references therein), we outline its short
proof for the reader's convenience.
   \begin{prop}  \label{cukia}
Let $F\colon \borel(\rbb) \to \ogr(\hh)$ be a
semispectral measure with compact support. Then
   \begin{equation*}
   \left[\begin{smallmatrix} I & \int_{\rbb} x
F(\D x) & \cdots & \int_{\rbb} x^n F(\D x)
   \\[1ex]
\int_{\rbb} x F(\D x) & \int_{\rbb} x^2 F(\D x)
& \cdots & \int_{\rbb} x^{n+1} F(\D x)
   \\
\vdots & \vdots & \ddots & \vdots
   \\[1ex]
\int_{\rbb} x^n F(\D x) & \int_{\rbb} x^{n+1}
F(\D x) & \cdots & \int_{\rbb} x^{2n} F(\D x)
   \end{smallmatrix} \right] \Ge 0,\quad n\in \zbb_+.
   \end{equation*}
   \end{prop}
   \begin{proof}
By Naimark's dilation theorem (see
Introduction), there exists a Hil\-bert space
$\kk$ containing $\hh$ and a spectral measure
$E\colon \borel(\rbb) \to \ogr(\kk)$ which
satisfies \eqref{fpeh} and \eqref{minimity}. By
Lemma~\ref{fuglemma}(iii), $E$ has compact
support. Applying the Stone-von Neumann
functional calculus, we obtain
   \begin{align*}
\sum_{j,k=0}^n \Big\langle \int_{\rbb} x^{j+k}
F(\D x) h_k, h_j\Big\rangle & =\sum_{j,k=0}^n
\Big\langle \int_{\rbb} x^{j+k} E(\D x) h_k,
h_j\Big\rangle
   \\
&= \Big\|\sum_{k=0}^n \int_{\rbb} x^k E(\D x)
h_k\Big\|^2 \Ge 0,
   \end{align*}
for all finite sequences $\{h_k\}_{k=0}^n
\subseteq \hh$.
   \end{proof}
   \begin{corollary} \label{istris}
Let $F\colon \borel(\rbb) \to \ogr(\hh)$ be a
semispectral measure with compact support. Then
$\mathrm{Var}(F) \Ge 0$, where $\mathrm{Var}(F)$
is as in \eqref{noise}.
   \end{corollary}
   \begin{proof}
Apply Proposition~\ref{cukia} with $n=2$ and use
the following well-known fact (see
\cite[Lemma~1]{Dav58}; see also
\cite[Theorem~5.1]{M-K-X19}): if $A, B \in
\ogr(\hh)$ are selfadjoint, $A$ is invertible in
$\ogr(\hh)$ and $X\in \ogr(\hh)$, then
$\left[\begin{smallmatrix} A & X
\\ X^* & B \end{smallmatrix}\right] \Ge 0$  if and
only if $B \Ge X^*A^{-1}X$.
   \end{proof}
   \section{\label{Sec.4}Proof of Theorem~\ref{poturs}}
Before proving the main result of this section,
we state the crucial lemma which seems to be of
some independent interest. We provide two proofs
of this lemma.
   \begin{lemma} \label{wkuwer}
Let $\mathcal{A}$ be a unital $C^*$-algebra,
$\varPhi\colon \mathcal{A} \to \ogr(\hh)$ be a
unital positive linear map and $a$ be a
selfadjoint element of $\mathcal{A}$. Then there
exists a unique semispectral measure $F\colon
\borel(\rbb) \to \ogr(\hh)$ such that $x^n \in
L^1(F)$ for all $n\in \zbb_+$ and
   \begin{align*}
\varPhi(a^n) = \int_{\rbb} x^n F(\D x), \quad
n\in \zbb_+.
   \end{align*}
Moreover, $F$ possesses the following
properties{\em :}
   \begin{enumerate}
   \item[(i)] $F$ has compact support,
   \item[(ii)] the closed support of $F$ is
contained in $\rbb_+$ whenever $a\Ge 0$.
   \end{enumerate}
   \end{lemma}
   \begin{proof}[First proof of Lemma~\ref{wkuwer}]
Replacing $\mathcal{A}$ by the unital
$C^*$-algebra generated by $\{a\}$, we may
assume without loss of generality that
$\mathcal{A}$ is commutative. Let $e$ denote the
unit of $\mathcal{A}$. According to
\cite[Corollary~2.9]{Paul86}, $\varPhi$ is
contractive and~therefore
   \begin{align} \label{surq2}
\|\varPhi(a^n)\| \Le \|a\|^n, \quad n\in \zbb_+.
   \end{align}
Since $a$ is selfadjoint, we see that
   \begin{align*}
   \left[\begin{smallmatrix} e & a^1& \cdots &
a^n
   \\[1ex]
a^1 & a^2& \cdots & a^{n+1}
   \\
\vdots & \vdots & \ddots & \vdots
   \\[1ex]
a^{n} & a^{n+1} & \cdots & a^{2n}
   \end{smallmatrix} \right]
   = \left[\begin{smallmatrix} e & a^1 & \cdots
& a^n
   \\[1ex]
0 & 0 & \cdots & 0
   \\
\vdots & \vdots & \ddots & \vdots
   \\[1ex]
0 & 0 & \cdots & 0
   \end{smallmatrix} \right]^*
   \left[\begin{smallmatrix} e & a^1 & \cdots &
a^n
   \\[1ex]
0 & 0 & \cdots & 0
   \\
\vdots & \vdots & \ddots & \vdots
   \\[1ex]
0 & 0 & \cdots & 0
   \end{smallmatrix} \right]  \Ge 0.
   \end{align*}
By the Stinespring theorem (see
\cite[Theorem~4]{Sti55}), $\varPhi$ is
completely positive, so
$[\varPhi(a^{j+k})]_{j,k=0}^n\Ge 0$. In
particular, we have
   \begin{align} \label{burd}
\sum_{j,k=0}^n \bar \lambda_j \lambda_k
\varPhi(a^{j+k}) \Ge 0, \quad
\{\lambda_j\}_{j=0}^n \subseteq \cbb, \, n \in
\zbb_+.
   \end{align}
Using \eqref{surq2} and \eqref{burd}, we deduce
from \cite[Theorem~2]{Bi94} that there exists a
semispectral measure $F\colon \borel(\rbb) \to
\ogr(\hh)$ such that
   \begin{align} \label{interp}
\langle\varPhi(a^n)h,h \rangle = \int_{\rbb} x^n
\langle F(\D x)h,h\rangle, \quad n \in \zbb_+,
\, h\in \hh.
   \end{align}
(that $F(\rbb)=I$ follows from the assumption
that $\varPhi$ is unital). Since $\lim_{r\to
\infty} \|f\|_r = \|f\|_{\infty}$ whenever
$\|f\|_r < \infty$ for some $r< \infty$ (see
\cite[Exercise~4, p.\ ~71]{Ru87}) and
   \begin{align*}
\lim_{n\to \infty} \Big(\int_{\rbb} x^{2n}
\langle F(\D x)h,h\rangle\Big)^{\frac{1}{2n}}
\overset{\eqref{interp}}= \lim_{n\to \infty}
\langle\varPhi(a^{2n})h, h
\rangle^{\frac{1}{2n}} \overset{\eqref{surq2}}
\Le \|a\|, \quad h \in \hh,
   \end{align*}
we deduce that
   \begin{align*}
\langle F(\{x\in \rbb\colon |x| >
\|a\|\})h,h\rangle=0, \quad h \in \hh.
   \end{align*}
Thus, the closed support of $F$ is contained in
$[-\|a\|, \|a\|]$. Combined with \eqref{form-ua}
and \eqref{interp}, this implies that $x^n \in
L^1(F)$ for all $n\in \zbb_+$ and
   \begin{align*}
\varPhi(a^n) = \int_{\rbb} x^n F(\D x), \quad n
\in \zbb_+.
   \end{align*}
Using \eqref{form-ua} and the well-known fact
that a Hamburger moment sequence having a
representing measure with compact support is
determinate (see \cite{Fug83}), we get the
uniqueness of $F$.

It remains to show that if $a \Ge 0$, then the
closed support of $F$ is contained in $\rbb_+$.
Using the square root theorem (see
\cite[Theorem~2.2.1]{Mur90}), we deduce that
   \begin{align*}
\left[\begin{smallmatrix}
a^1 & a^2 & \cdots &
a^{n+1}
   \\[1ex]
a^2 & a^3& \cdots & a^{n+2}
   \\
\vdots & \vdots & \ddots & \vdots
   \\[1ex]
a^{n+1} & a^{n+2} & \cdots & a^{2n+1}
   \end{smallmatrix} \right] =
\left[\begin{smallmatrix} a^{\frac{1}{2}} &
a^{\frac{3}{2}} & \cdots & a^{\frac{2n+1}{2}}
   \\[1ex]
0 & 0 & \cdots & 0
   \\
\vdots & \vdots & \ddots & \vdots
   \\[1ex] 0 & 0 & \cdots & 0
   \end{smallmatrix} \right]^*
\left[\begin{smallmatrix} a^{\frac{1}{2}} &
a^{\frac{3}{2}} & \cdots & a^{\frac{2n+1}{2}}
   \\[1ex]
0 & 0 & \cdots & 0
   \\
\vdots & \vdots & \ddots & \vdots
   \\[1ex]
0 & 0 & \cdots & 0
   \end{smallmatrix} \right]  \Ge 0.
   \end{align*}
Hence, by \cite[Theorem~4]{Sti55},
$[\varPhi(a^{j+k+1})]_{j,k=0}^n\Ge 0$, which
implies that
   \begin{align} \label{bure}
\sum_{j,k=0}^n \bar \lambda_j \lambda_k
\varPhi(a^{j+k+1}) \Ge 0, \quad
\{\lambda_j\}_{j=0}^n \subseteq \cbb, \, n \in
\zbb_+.
   \end{align}
Combining \eqref{burd}, \eqref{bure} and the
Stieltjes theorem (see
\cite[Theorem~6.2.5]{B-C-R}) with the uniqueness
of $F$, we conclude that the closed support of
$F$ is contained in $\rbb_+$.
   \end{proof}
   \begin{proof}[Second proof of Lemma~\ref{wkuwer}]
As in the first proof of Lemma~\ref{wkuwer},
there is no loss of generality in assuming that
$\mathcal{A}$ is commutative. By the Stinespring
dilation theorem (see \cite[Theorems~1 and
4]{Sti55}), there exist a Hilbert space $\kk$
containing $\hh$ and a $*$-representation
$\pi\colon \mathcal{A}\to\ogr(\kk)$ such that
   \begin{align}  \label{wizu1}
\varPhi(u) = P\pi(u)|_{\hh}, \quad u \in
\mathcal{A},
   \end{align}
where $P\in \ogr(\kk)$ is the orthogonal
projection of $\kk$ onto $\hh$. Applying
\cite[Theorem~12.22]{Rud73}, we deduce that
there exists a spectral measure $E\colon
\borel(\mathfrak M) \to \ogr(\kk)$ such~that
  \begin{align} \label{wizu1.5}
\pi(u) = \int_{\mfr} \widehat{\pi(u)} \, \D E,
\quad u\in \mathcal{A},
  \end{align}
where $\mfr$ is the maximal ideal space of the
unital commutative $C^*$-algebra
$\overline{\pi(\mathcal{A})}$, the (operator
norm) closure of $\pi(\mathcal{A})$ in
$\ogr(\kk)$, and $\widehat{\pi(u)}\colon \mfr
\to \cbb$ is the Gelfand transform of $\pi(u)$.
Set $M =PE|_{\hh}$. It follows from
\eqref{wizu1} and \eqref{wizu1.5} that
   \begin{align} \label{wizu2}
\varPhi(u) = \int_{\mfr} \widehat{\pi(u)} \D M,
\quad u\in \mathcal{A}.
  \end{align}
Define the semispectral measure $F \colon
\borel(\rbb) \to \ogr(\hh)$ by
   \begin{align*}
F(\varDelta) = M
\big(\widehat{\pi(a)}{}^{-1}(\varDelta)\big),
\quad \varDelta \in \borel(\rbb).
   \end{align*}
By \cite[Theorem~11.18]{Rud73} and the
assumption that $a=a^*$, we see that
$\widehat{\pi(a)} \colon \mfr \to \rbb$. Since
$\widehat{\pi(a)}$ is continuous and $\mfr$ is a
compact Hausdorff space, we deduce that
$\widehat{\pi(a)}(\mfr)$ is a compact subset of
$\rbb$ such that $F\big(\rbb \backslash
\widehat{\pi(a)}(\mfr)\big)=0$, which implies
that the semispectral measure $F$ has compact
support. Applying \eqref{form-ua} and the
measure transport theorem (cf.\
\cite[Theorem~1.6.12]{Ash00}), we conclude that
   \begin{align*}
\varPhi(a^n) \overset{\eqref{wizu2}}=
\int_{\mfr} \widehat{\pi(a)}{}^n \D M =
\int_{\rbb} x^n F (\D x), \quad n\in \zbb_+.
  \end{align*}
The proof of the uniqueness of $F$ proceeds as
before.

Finally, if $a\Ge 0$, then by the square root
theorem and \cite[Theorem~11.18]{Rud73}, we
deduce that
   \begin{align*}
\widehat{\pi(a)} =
\widehat{\pi(a^{\frac{1}{2}})}{}^2 \Ge 0,
   \end{align*}
which implies that the closed support of $F$ is
contained in $\rbb_+$.
   \end{proof}
   At this point we are ready to prove the main
result of this section.
   \begin{proof}[Proof of Theorem~\ref{poturs}]
   (i)$\Rightarrow$(ii) Since the map $\varPhi$
preserves selfadjointness, $b:=\varPhi(a)$ does
the job.

   (ii)$\Rightarrow$(i) In view of the
Gelfand-Naimark theorem (see
\cite[Theorem~12.41]{Rud73}), there is no loss
of generality in assuming that
$\mathcal{B}=\ogr(\hh)$. By Lemma~\ref{wkuwer},
there exists a semispectral measure $F\colon
\borel(\rbb) \to \ogr(\hh)$ with compact support
such that
   \begin{align} \label{cgfre}
\varPhi(a^n) = \int_{\rbb} x^n F(\D x), \quad
n\in \zbb_+.
   \end{align}
Therefore, by (ii), we have
   \begin{align*}
b^k=\varPhi(a^k) \overset{\eqref{cgfre}}=
\int_{\rbb} x^k F (\D x), \quad k=p,q.
   \end{align*}
Applying Theorem~\ref{main1} to $T=b$, we
conclude that $F$ is the spectral measure of
$b$. Using the Stone-von Neumann functional
calculus, we get
   \begin{align*}
\varPhi(a^n) \overset{\eqref{cgfre}}=
\Big(\int_{\rbb} x F(\D x)\Big)^n = b^n, \quad n
\in\zbb_+,
   \end{align*}
so $b=\varPhi(a)$, which yields
$\varPhi(a^n)=\varPhi(a)^n$ for all $n\in
\zbb_+$. This implies (i).
   \end{proof}
   \begin{rem} \label{takisubi}
Theorem~\ref{poturs} is somewhat related to a
result of D. Petz (see \cite[Theorem]{Petz86})
which shows that equality holds in Jensen's
inequality $f(\varPhi(a)) \Le \varPhi(f(a))$ if
and only if $\varPhi$ restricted to the unital
subalgebra generated by $a$ is multiplicative,
where $\varPhi$ is a unital positive linear map
between unital $C^*$-algebras, $f$ is a
non-affine operator convex function on an open
subinterval $J$ of $\rbb$ and $a$ is a
selfadjoint element with spectrum in $J$. The
main difference between Petz's result and
Theorem~\ref{poturs} is that the monomial $x^n$
with $n\in \zbb_+$ is a non-affine operator
convex function on $J$ if and only if $n = 2$
(apply rescaling and
\cite[Exercise~V.2.11]{Bha97}, see also
\cite{Riz-Sho79}).
   \hfill $\diamondsuit$
   \end{rem}
   \begin{rem} \label{rown}
We have deduced Theorem~\ref{poturs} from
Theorem~\ref{main1}. It turns out that these two
results are logically equivalent. Indeed, under
the assumptions and notation of
Theorem~\ref{main1}, it suffices to show that
(iii) implies (i). For, define the unital
positive linear map $\varPhi\colon C(K) \to
\ogr(\hh)$ by
   \begin{align*}
\varPhi(f) = \int_{K} f(x) F(\D x), \quad f\in
C(K),
   \end{align*}
where $K$ stands for the closed support of $F$.
Let $a\in C(K)$ be the function defined by
$a(x)=x$ for $x\in K$ and let $b=T$. Using
Theorem~\ref{poturs}, we deduce that $T=\int_{K}
x F(\D x)$ and
   \begin{align*}
\int_{K} x^n F(\D x) = T^n = \int_{K} x^n G(\D
x), \quad n\in \zbb_+,
   \end{align*}
where $G$ is the spectral measure of $T$. Using
\eqref{form-ua} and the well-known fact that a
Hamburger moment sequence having a representing
measure with compact support is determinate, we
conclude that $F=G$, which completes the proof.

A careful inspection of the proof of
Theorem~\ref{poturs} in conjunction with the
above discussion shows that in fact
Problems~\ref{momentprob} and \ref{algebraprob}
are logically equivalent regardless of the
cardinality of the set $\varXi$.
   \hfill $\diamondsuit$
   \end{rem}
In case where the elements $a$ and $b$ are
positive, we get the following version of
Theorem~\ref{poturs}.
   \begin{theorem}
Suppose that $\mathcal{A}$ and $\mathcal{ B}$
are unital $C^*$-algebras, $\varPhi \colon
\mathcal{A} \rightarrow \mathcal{B}$ is a unital
positive linear map, $a\in \mathcal{A}$ is
positive and $p, q$ are distinct positive
integers. Then the following conditions are
equivalent{\em :}
   \begin{enumerate}
   \item[(i)] $\varPhi$ restricted to
the unital subalgebra generated by $\{a\}$ is
multiplicative,
   \item[(ii)] there exists a positive
element $b\in \mathcal{B}$ such that
$b^k=\varPhi(a^k)$ for $k=p,q$.
   \end{enumerate}
Moreover, if {\em (ii)} holds, then
$b=\varPhi(a)$.
   \end{theorem}
   \begin{proof}
It suffices to show the implication
(ii)$\Rightarrow$(i). We present two proofs. The
first one is the same as the proof of
Theorem~\ref{poturs}, with the only difference
that we use Theorem~\ref{maintw4} instead of
Theorem~\ref{main1}. We leave the details to the
reader.

The second proof relies upon Petz's result.
Without loss of generality, we may assume that
$0 < p < q$. Let $f\colon [0,\infty) \to \rbb$
be the function given by $f(x)=-x^{\frac{p}{q}}$
for $x\in [0,\infty)$. It follows from
\cite[Theorems~V.1.9 and V.2.5]{Bha97} that $f$
is an operator convex function. Using (ii) and
the Stone-von Neumann functional calculus, we
get
   \begin{align*}
f(\varPhi(a^{q})) = -
\varPhi(a^{q})^{\frac{p}{q}} = -
(b^{q})^{\frac{p}{q}} = - b^{p},
   \end{align*}
and
   \begin{align*}
\varPhi(f(a^{q})) = -
\varPhi((a^{q})^{\frac{p}{q}}) = -
\varPhi(a^{p}) = - b^{p}.
   \end{align*}
Consequently,
  \begin{equation*}
f(\varPhi(a^{q}))=\varPhi(f(a^{q})).
  \end{equation*}
Combined with \cite[Theorem]{Petz86} and the
fact that $\varPhi$ is continuous (see
\cite[Corollary~2.9]{Paul86}), this implies that
$\varPhi$ restricted to the unital $C^*$-algebra
generated by $\{a^{q}\}$ is multiplicative.
Applying the Stone-von Neumann functional
calculus and the Weierstrass approximation
theorem, one can show that the unital
$C^*$-algebras generated by $\{a\}$ and
$\{a^{q}\}$ coincide (this is a very special
case of the M\"{u}ntz-Sz\'{a}sz theorem, see
\cite[Theorem~15.26]{Ru87}). Hence, (i) holds.
   \end{proof}
   \section{\label{Sec.5}Proof of Theorem~\ref{main2-w0}}
We begin with a simple observation related to
Problems~\ref{momentprob} and
\ref{dilationprob}. Namely, if $T$ and $F$
satisfy \eqref{momprob}, then by \eqref{form-ua}
and the measure transport theorem for every
$\tau\in \rbb\backslash \{0\}$, $\tau T$ and
$F_\tau$ satisfy \eqref{momprob}, where
$F_\tau\colon \borel(\rbb)\to \ogr(\hh)$ is the
semispectral measure with compact support given
by
   \begin{align*}
F_\tau(\varDelta) = F(\tau^{-1} \varDelta),
\quad \varDelta \in \borel(\rbb).
   \end{align*}
Moreover, $F$ is spectral if and only if
$F_\tau$ is spectral. A similar observation
applies to Problem~\ref{dilationprob}. In other
words, rescaling preserves the affirmative or
negative solutions to Problems~\ref{momentprob}
and \ref{dilationprob}.

Next we prove a lemma that is central to the
proof of Theorem~\ref{main2}.
   \begin{lemma} \label{uklem}
Suppose that $(p,q)\in
\natu^2\backslash\varOmega$ and $p\Le q$, where
$\varOmega$ is as in \eqref{omigra}. Then for
every $\tau \in \rbb \backslash \{0\}$, there
exist $\alpha,\beta \in (0,1)$ and distinct
$\lambda_1, \lambda_2\in \rbb$ such~that
   \begin{equation}\label{uklrow}
\left\{ \begin{array}{ll} \alpha+\beta=1,
   \\[.5ex]
\alpha\lambda_1^p+\beta\lambda_2^p=\tau^p,
   \\[.5ex]
\alpha\lambda_1^q+\beta\lambda_2^q=\tau^q.
   \end{array}
   \right.
   \end{equation}
   \end{lemma}
   \begin{proof}
We may assume, without loss of generality, that
$\tau=1$. Using the substitution
   \begin{equation*}
\alpha=\frac{a}{a+b}\quad\text{and}\quad\beta=\frac{b}{a+b}
   \end{equation*}
with $a,b \in (0,\infty)$, we obtain an
equivalent system of equations:
   \begin{equation*}
\left\{ \begin{array}{ll} \frac{a}{a+b}
\lambda_1^p + \frac{b}{a+b} \lambda_2^p = 1,
   \\[.5ex]
\frac{a}{a+b} \lambda_1^q + \frac{b}{a+b}
\lambda_2^q=1.
   \end{array}
   \right.
   \end{equation*}
   Multiplying both sides of the above
equalities by $a+b$ and rearranging gives
   \begin{equation}\label{sysmateq}
\left\{ \begin{array}{ll}
{a}(\lambda_1^p-1)+{b}(\lambda_2^p-1)=0,
   \\[.5ex]
{a}(\lambda_1^q-1)+{b}(\lambda_2^q-1)=0.
   \end{array}
   \right.
   \end{equation}
The determinant of the above system of equations
(with unknowns $a,b$) is
   \begin{equation*}
D(\lambda_1,\lambda_2) =\det \left[
   \begin{matrix}
\lambda_1^p-1 & \lambda_2^p-1
   \\[.5ex]
\lambda_1^q-1 & \lambda_2^q-1
   \end{matrix}
   \right] = (\lambda_1^p-1) (\lambda_2^q-1)-(
\lambda_2^p-1)( \lambda_1^q-1).
   \end{equation*}
Observe that if $D(\lambda_1,\lambda_2) \neq 0$,
then the system of equations \eqref{sysmateq}
has only one solution $a=b=0$. Thus, the only
chance to find nonzero solutions $a, b$ of the
system \eqref{sysmateq} is when
$D(\lambda_1,\lambda_2)=0$. Note that
   \begin{equation}\label{eqd}
\text{if $\lambda_1,\lambda_2\in
\rbb\backslash\{-1,1\}$, then
$D(\lambda_1,\lambda_2)=0$ if and only if
$\frac{\lambda_1^p-1}{\lambda_1^q-1} =
\frac{\lambda_2^p-1}{\lambda_2^q-1}$.}
   \end{equation}

We will consider four cases.

{\sc Case 1.} $p=q$.

It is easily seen that for any $\alpha, \beta\in
(0,1)$ such that $\alpha + \beta =1$, there are
plenty of two-element subsets
$\{\lambda_1,\lambda_2\}$ of $(0,\infty)$
solving \eqref{uklrow}.

{\sc Case 2.} Both $p$ and $q$ are even.

Set $\lambda_1=-1$, $\lambda_2=1$ and
$\beta=1-\alpha$, where $\alpha\in (0,1)$. Then
it is easily seen that \eqref{uklrow} is
satisfied.

{\sc Case 3.} $p<q$, $p$ even and $q$ odd.

Consider the function $\phi \colon
[0,1]\to[0,1]$ defined by
   \begin{equation*}
\phi(x)=\frac{1-x^p}{1+x^q}, \quad x\in[0,1].
   \end{equation*}
Then $\phi$ is continuous, $\phi(0)=1$,
$\phi(1)=0$ and $\phi((0,1)) \subseteq (0,1)$.
By the Darboux property of continuous functions,
$\phi((0,1))=(0,1)$. Take $\lambda_1>1$. Then
   \begin{equation*}
\frac{1-\lambda_1^p}{1-\lambda_1^q}\in (0,1).
   \end{equation*}
Therefore, there exists $x\in (0,1)$ such that
   \begin{equation}\label{dorb}
\frac{1-\lambda_1^p}{1-\lambda_1^q}=\phi(x).
   \end{equation}
Set $\lambda_2=-x$. Then $\lambda_2<0$,
$|\lambda_2|=x<1$ and
   \begin{equation*}
\frac{1-\lambda_1^p}{1-\lambda_1^q}
\overset{\eqref{dorb}}{=} \frac{1-|\lambda_2|^p}
{1+|\lambda_2|^q} =
\frac{1-\lambda_2^p}{1-\lambda_2^q},
   \end{equation*}
which by \eqref{eqd} means that
$D(\lambda_1,\lambda_2)=0$, so the system of
equations \eqref{sysmateq} is linearly
dependent. Take any $a\in (0,\infty)$ and set
   \begin{equation*}
b=a\frac{\lambda_1^p-1}{1-\lambda_2^p}>0.
   \end{equation*}
Then the pair $(a,b)$ is a solution of the
system of equations \eqref{sysmateq}.

{\sc Case 4.} $p<q$ and both $p$ and $q$ are
odd.

Consider the function $\psi\colon
[1,\infty)\to(0,1]$ defined by
   \begin{equation*}
\psi(x)=\frac{1+x^p}{1+x^q},\quad
x\in[1,\infty).
   \end{equation*}
Then $\psi$ is continuous, $\psi(1)=1$,
$\lim_{x\to\infty}\psi(x)=0$ and
$\psi((1,\infty))\subseteq(0,1)$. As a
consequence of the Darboux property of
continuous functions, $\psi((1,\infty))=(0,1)$.
Take $\lambda_1>1$. Observe that
   \begin{equation*}
\frac{1-\lambda_1^p}{1-\lambda_1^q}\in (0,1).
   \end{equation*}
Hence, there exists $x\in(1,\infty)$ such that
   \begin{equation} \label{dorb2}
\frac{1-\lambda_1^p}{1-\lambda_1^q}=\psi(x).
   \end{equation}
Set $\lambda_2:=-x$. Then $\lambda_2<0$,
$|\lambda_2|=x>1$ and
   \begin{equation*}
\frac{1-\lambda_1^p} {1-\lambda_1^q}
\overset{\eqref{dorb2}}{=} \psi(|\lambda_2|) =
\frac{1+|\lambda_2|^p} {1+|\lambda_2|^q} =
\frac{1-\lambda_2^p}{1-\lambda_2^q}.
   \end{equation*}
As in Case~3, taking any $a\in (0,\infty)$ and
setting
   \begin{equation*}
b=a\frac{\lambda_1^p-1}{1-\lambda_2^p} = a
\frac{\lambda_1^p-1} {1+|\lambda_2|^p}>0,
   \end{equation*}
we see that the pair $(a,b)$ is a solution of
the system of equations \eqref{sysmateq}. This
completes the proof.
   \end{proof}
We are now ready to prove the main result of
this section, which provides the
counter-examples mentioned earlier in
Introduction. In fact, this is a stronger
version of Theorem~\ref{main2-w0}.
   \begin{theorem} \label{main2}
Suppose that $(p,q)\in
\natu^2\backslash\varOmega$ and $p\Le q$, where
$\varOmega$ is as in \eqref{omigra}. Let $\tau
\in \rbb\backslash \{0\}$. Set $\hh=\cbb$ and
$T=\tau I$. Then there exist $\alpha,\beta \in
(0,1)$ and $\lambda_1, \lambda_2\in \rbb$ such
that $\alpha+\beta=1$, $\lambda_1 \neq
\lambda_2$ and
   \begin{itemize}
   \item[(i)] the semispectral
measure $F\colon \borel(\rbb) \to \ogr(\hh)$
defined by
   \begin{equation} \label{gdfer}
F(\varDelta) = \alpha\delta_{\lambda_1}
(\varDelta) I + \beta \delta_{\lambda_2}
(\varDelta)I, \quad \varDelta \in \borel(\rbb),
   \end{equation}
is not spectral and
   \begin{equation} \label{tritk}
T^k=\int_{\rbb} x^k F(\D x), \quad k=p,q,
   \end{equation}
   \item[(ii)]
the selfadjoint operator $S\in\ogr(\hh\oplus
\hh)$ defined by
   \begin{equation*}
S = \left[ \begin{matrix} \alpha \lambda_1 +
\beta \lambda_2 & \sqrt{\alpha\beta} (\lambda_1
- \lambda_2)
   \\[1ex]
\sqrt{\alpha\beta}(\lambda_1-\lambda_2) &
\beta\lambda_1+\alpha\lambda_2
   \end{matrix} \right]
   \end{equation*}
does not commute with $P:=\big[
\begin{smallmatrix} 1 & 0 \\0 & 0
\end{smallmatrix}\big]$ and
   \begin{align}\label{dil}
T^k = P S^k|_{\hh}, \quad k=p,q.
   \end{align}
   \item[(iii)] the unital positive  linear map $\varPhi \colon
C(K) \to \ogr(\hh)$ defined by
   \begin{align} \label{konkr}
\varPhi(f) = \int_{K} f \D F, \quad f \in C(K),
   \end{align}
is not multiplicative, $C(K)$ is the unital
algebra generated by $a$ and
   \begin{align} \label{simaduff}
b^k=\varPhi(a^k), \quad k=p,q,
   \end{align}
where $F$ is as in \eqref{gdfer},
$K=\{\lambda_1,\lambda_2\}$, $a(x)=x$ for $x\in
K$ and $b= T$.
   \end{itemize}
   \end{theorem}
   \begin{proof}
Let $\alpha,\beta,\lambda_1,\lambda_2$ be as in
Lemma~\ref{uklem}.

(i) By \eqref{uklrow} and \eqref{gdfer}, $F$ is
a semispectral measure satisfying \eqref{tritk}.
However, $F$ is not a spectral measure because
$F(\{\lambda_1\})=\alpha\in(0,1)$.

(ii) Clearly, the operator $S$ is selfadjoint.
It follows from the first equality in
\eqref{uklrow} that the matrix
   \begin{equation*}
\left[ \begin{matrix} \sqrt{\alpha} &
\sqrt{\beta}
   \\[1ex]
-\sqrt{\beta} & \sqrt{\alpha}
\end{matrix}
   \right]
   \end{equation*}
is unitary and consequently
   \begin{equation}  \label{jordun}
\left[ \begin{matrix} \sqrt{\alpha} &
\sqrt{\beta}
   \\[1ex]
-\sqrt{\beta} & \sqrt{\alpha}
\end{matrix}
   \right]^{-1} = \left[ \begin{matrix}
   \sqrt{\alpha} & -\sqrt{\beta}
   \\[1ex]
\sqrt{\beta} & \sqrt{\alpha}
\end{matrix} \right].
   \end{equation}
Now it is easily seen that the Jordan
decomposition of $S$ takes the form
   \begin{equation*}
    S =
\left[ \begin{matrix}
\sqrt{\alpha} & -\sqrt{\beta}
   \\[1ex]
\sqrt{\beta} &
\sqrt{\alpha}
\end{matrix} \right]
\left[ \begin{matrix}
\lambda_1 & 0
   \\[1ex]
0 & \lambda_2
\end{matrix} \right]
\left[ \begin{matrix}
\sqrt{\alpha} & \sqrt{\beta}
   \\[1ex]
-\sqrt{\beta} & \sqrt{\alpha}
   \end{matrix}
   \right].
   \end{equation*}
Combined with \eqref{jordun}, this implies that
   \begin{equation}\label{pows}
    S^n =
\left[ \begin{matrix}
\alpha\lambda_1^n+\beta\lambda_2^n &
\sqrt{\alpha\beta} (\lambda_1^n-\lambda_2^n)
   \\[1ex]
\sqrt{\alpha\beta}(\lambda_1^n-\lambda_2^n) &
\beta\lambda_1^n+\alpha\lambda_2^n
\end{matrix}
   \right], \quad n \in \zbb_+.
   \end{equation}
By \eqref{uklrow} and \eqref{pows}, the condition
\eqref{dil} is satisfied. Since $\lambda_1\neq
\lambda_2$, the operator $S$ does not commute with
$P$.

(iii) It is immediate from \eqref{konkr} and (i)
that $\varPhi$ is the unital positive linear map
which satisfies \eqref{simaduff}. Clearly,
$C(K)$ is the unital algebra generated by $a$.
To show that $\varPhi$ is not multiplicative,
consider two polynomials $u(x) = x-\lambda_1$
and $v(x)=x-\lambda_2$ and note that
$\varPhi(u(a)v(a))=0$ while $\varPhi(u(a))
\varPhi(v(a)) = - \alpha\beta (\lambda_1 -
\lambda_2)^2 \neq 0$. This completes the proof.
   \end{proof}
   \begin{rem}
A careful inspection of the proofs of
Lemma~\ref{uklem} and Theorem~\ref{main2}
reveals that there is a lot of freedom in
choosing the parameters $\lambda_1$ and
$\lambda_2$. Note that if $p < q$ and $\tau >
0$, then in view of \cite[Theorem~4.2]{P-S}, at
least one of the parameters $\lambda_1$ or
$\lambda_2$ must be negative. Note also that if
$p<q$, $p$ is odd and $q$ is even, then
according to Theorem~\ref{main1} there are no
$\alpha,\beta, \lambda_1, \lambda_2$ satisfying
the conclusion of Theorem \ref{main2}. However,
in this particular case, we can justify it in an
elementary way. Namely, by the H\"older
inequality, we infer from \eqref{uklrow} that
   \begin{align} \notag
|\tau|^p= |\alpha\lambda_1^p+\beta\lambda_2^p| &
\Le \alpha|\lambda_1|^p+\beta|\lambda_2|^p
      \\  \notag
& \Le \sqrt[r]{\alpha+\beta}\sqrt[q/p]{
\alpha|\lambda_1|^q+\beta|\lambda_2|^q}
   \\  \label{gnu2}
& = \sqrt[r]{\alpha+\beta}\sqrt[q/p]{
\alpha\lambda_1^q+\beta\lambda_2^q}= |\tau|^p,
   \end{align}
where $r\in (1,\infty)$ is such that
$\frac{1}{r}+\frac{p}{q}=1$. This means that
equality in the H\"older inequality holds. As a
consequence, we deduce that
$|\lambda_1|=|\lambda_2|$. Combined with the
first inequality in \eqref{gnu2} and the
assumption that $p$ is odd, this implies that
$\lambda_1=\lambda_2$, which is a contradiction.
   \hfill $\diamondsuit$
   \end{rem}
   \section{\label{Sec.6}More examples}
In this section we illustrate
Theorem~\ref{main2-w0} (cf.\
Theorem~\ref{main2}) by considering two
interesting examples for the case where $p=2$
and $q=3$. Now we use the dilation approach as
stated in Lemma~\ref{fuglemma} (cf.\
Problems~\ref{momentprob} and
\ref{dilationprob}).
   \begin{ex}  \label{prsz1}
Let $\{f_n\}_{n=0}^\infty$ be the Fibonacci
sequence, that is, $f_0=0$, $f_1=1$ and
$f_{n+1}=f_n+f_{n-1}$ for $n\in \natu$. It is
well known and easy to prove that
   \begin{equation} \label{fubonn}
   \left[ \begin{matrix}
0 & 1 \\
1 & 1
\end{matrix} \right]^n=\left[ \begin{matrix}
f_{n-1} & f_{n}\\
f_{n} & f_{n+1}
\end{matrix}
   \right], \quad n \in \mathbb{N}.
   \end{equation}
Set $\hh=\cbb$, $T=I$ and
   \begin{equation*}
   S= \left[ \begin{matrix}
0 & 1 \\
1 & 1
   \end{matrix}
   \right].
   \end{equation*}
Then $S$ is a selfadjoint operator and the
spectral measure $E$ of $S$ is given by
   \begin{align*}
E(\varDelta) = \frac{1}{1+\phi^2}
\delta_{1-\phi} (\varDelta) \left[
\begin{matrix}
\phi^2 & - \phi \\
- \phi & 1
\end{matrix} \right] + \frac{1}{1+\phi^2}
\delta_{\phi} (\varDelta) \left[
\begin{matrix}
1 & \phi \\
\phi & \phi^2
\end{matrix} \right], \quad \varDelta \in
\borel(\rbb),
   \end{align*}
where $\phi=\frac{1+\sqrt{5}}{2}$ is the golden
ratio. It follows from \eqref{fubonn} that
   \begin{align*}
   \begin{minipage}{70ex}
$T^k = P S^k|_{\hh}$ for $k=2,3$ and $T^k \neq P
S^k|_{\hh}$ for $k\in \natu\backslash\{2,3\}$,
   \end{minipage}
   \end{align*}
where $P= \left[ \begin{smallmatrix}
1 & 0 \\
0 & 0
   \end{smallmatrix}
   \right]$ is the orthogonal projection of $\hh
\oplus \hh$ onto $\hh$. The semispectral measure
$F:=PE|_{\hh}$ (see \eqref{fpeh} for the
definition of $PE|_{\hh}$) takes the form
   \begin{equation} \label{cofjus}
F(\varDelta) = \frac{\phi^2}{1+\phi^2}
\delta_{1-\phi} (\varDelta)I +
\frac{1}{1+\phi^2} \delta_{\phi} (\varDelta) I,
\quad \varDelta \in \borel(\rbb).
   \end{equation}
By Lemma~\ref{fuglemma}(ii), we have
   \begin{align*}
   \begin{minipage}{70ex}
$T^k = \int_{\rbb} x^k F(\D x)$ for $k=2,3$ and
$T^k \neq \int_{\rbb} x^k F(\D x)$ for $k\in
\natu\backslash\{2,3\}$.
   \end{minipage}
   \end{align*}
Clearly, by \eqref{cofjus}, $F$ is not a
spectral measure (also because $P$ does not
commute with $S$, see Lemma~\ref{fuglemma}).
   \hfill $\diamondsuit$
   \end{ex}
The next example is a modification of the
previous one.
   \begin{ex} \label{dr-ezam}
Let $T\in \ogr(\hh)$ be a nonzero selfadjoint
operator and let $S\in \ogr(\hh\oplus\hh)$ be
the selfadjoint operator given by the $2 \times
2$ block matrix
   \begin{equation} \label{nocku}
S= \left[ \begin{matrix}
0 & T \\
T & T
\end{matrix} \right].
   \end{equation}
One can verify that
   \begin{align*}
S^2=\left[ \begin{matrix}
T^2 & T^2 \\
T^2 & 2T^2
\end{matrix} \right]
\quad \text{and} \quad S^3=\left[
\begin{matrix}
T^3 & 2T^3 \\
2T^3 & 3T^3
\end{matrix} \right].
   \end{align*}
Thus the operators $T$ and $S$ satisfy the
following two identities:
   \begin{align} \label{hrumhre}
T^k = P S^k|_{\hh}, \quad k=2,3,
   \end{align}
where $P= \left[ \begin{smallmatrix}
I_{\hh} & 0 \\
0 & 0
   \end{smallmatrix}
   \right]$
   is the orthogonal projection of $\hh\oplus
\hh$ onto $\hh$. Let $E$ be the spectral measure
of $S$ and let $F:=PE|_{\hh}$ be the
corresponding semispectral measure. Applying
Lemma~\ref{fuglemma}(ii) and using
\eqref{hrumhre}, we get
   \begin{align*}
T^k = \int_{\rbb} x^k F(\D x), \quad k=2,3.
   \end{align*}
Since $T\neq 0$, the operator $P$ does not
commute with $S$, so by Lemma~\ref{fuglemma},
$F$ is not a spectral measure. In contrast to
Example~\ref{prsz1}, here it is much easier to
use Lemma~\ref{fuglemma} to see that the
semispectral measure $F$ is not spectral.
   \hfill $\diamondsuit$
   \end{ex}
   \begin{rem}
Regarding Example~\ref{dr-ezam}, note that the
operator $S$ given by \eqref{nocku} is unitarily
equivalent to the tensor product
   \begin{align*}
S=\left[ \begin{matrix}
0 & 1 \\
1 & 1
\end{matrix} \right] \otimes T.
   \end{align*}
Combined with \eqref{fubonn}, this implies that
   \begin{align*}
S^n=\left[ \begin{matrix}
0 & 1 \\
1 & 1
\end{matrix} \right]^n \otimes T^n = \left[ \begin{matrix}
f_{n-1} T^n & f_{n} T^n \\
f_{n} T^n & f_{n+1} T^n
\end{matrix}
   \right], \quad n \in \natu.
   \end{align*}
This means that tensoring and orthogonal
summation enrich the class of counterexamples by
allowing semispectral measures to have operator
values on Hilbert spaces of arbitrary dimension.
   \hfill $\diamondsuit$
   \end{rem}
   \section{\label{Sec.7}Semispectral measures with non-compact supports}
In this section we extend the two-moment
characterizations of spectral measures given in
Theorems~\ref{maintw4} and \ref{main1} to the
case of semispectral measures with non-compact
supports. When considering a Borel semispectral
measure $F$ on the real line with non-compact
support, it may happen that the coordinate
function $\rbb \ni x \mapsto x \in \rbb$ is not
in $L^1(F)$ (cf.\ \eqref{form-ua}), which means
that the expression $\int_{\rbb} x^k F(\D x)$,
where $k\in \natu$, may not yield a bounded
operator. On the other hand, if
Problem~\ref{momentprob} has an affirmative
solution under the formally weaker assumption
that the functions $\rbb \ni x \mapsto x^k \in
\rbb$, $k \in \varXi$, are in $L^1(F)$, then $F$
being {\em a posteriori} a spectral measure must
have a compact support. Indeed, by the measure
transport theorem $T^k = \int_{\rbb} x (F\circ
\varphi_k^{-1})(\D x)$, where $\varphi_k$ is as
in \eqref{fiub} and $k$ is any element of
$\varXi$, and thus $F\circ \varphi_k^{-1}$ is
the spectral measure of the (bounded)
selfadjoint operator $T^k$. Hence, $F\circ
\varphi_k^{-1}$ must have compact support (see
\cite[Theorem~5.9]{Sch12}). Consequently, $F$
itself must have compact support. A similar
argument applied in the case where a
semispectral measure $F\colon \ascr \to
\ogr(\hh)$ is considered on an abstract
measurable space $(X,\ascr)$, and the coordinate
function is replaced by a measurable real-valued
function $\omega$ on $X$, leads to the
conclusion that $\omega$ is $F$-essentially
bounded, that is, $F(\{x\in X\colon |\omega(x)|>
r\})=0$ for some $r\in \rbb_+$ (equivalently,
$\omega \in L^{\infty}(F)$). However, to get the
spectrality of $F$, it is not enough to assume
that $\omega$ is $F$-essentially bounded. It
turns out that the ``missing'' property of
$\omega$ is $\sigma$-surjectivity. We say that a
measurable map $f\colon X \to Y$ between
measurable spaces $(X, \ascr)$ and $(Y,\bscr)$
(i.e., a map such that $f^{-1}(\varDelta) \in
\ascr$ for all $\varDelta \in \bscr$) is {\em
$\sigma$-surjective} if the corresponding map
$\bscr \ni \varDelta \longmapsto
f^{-1}(\varDelta) \in \ascr$ is surjective. If
$Y$ is a topological Hausdorff space,
$\sigma$-surjectivity refers to
$\bscr=\borel(Y)$. In case $\ascr=\borel(X)$ and
$\bscr=\borel(Y)$, where $X$ and $Y$ are
topological Hausdorff spaces,
$\sigma$-surjectivity is called in
\cite{C-S-s19} {\em Borel injectivity}. It is
worth emphasizing here that the property of
being $\sigma$-surjective was used in the proof
of Theorem~\ref{main1}.

Applying the measure transport theorem together
with Theorems~\ref{maintw4} and \ref{main1}, we
get the following.
   \begin{theorem} \label{notcump}
Assume that $(X, \ascr)$ is a measurable space,
$F\colon \ascr \to \ogr(\hh)$ is a semispectral
measure, $T \in \ogr(\hh)$ is a selfadjoint
operator and $p,q$ are positive integers such
that $p<q$. Let $\omega\colon X\to \rbb$ be an
$F$-essentially bounded $\sigma$-surjective
function such that
   \begin{align*}
T^k =\int_X \omega(x)^k F(\D x), \quad k=p, q.
   \end{align*}
If $p$ is odd and $q$ is even, or if $\omega(X)
\subseteq \rbb_+$, then $F$ is a spectral
measure.
   \end{theorem}
It is worth pointing out that there is a wide
class of measurable spaces admitting functions
$\omega$ with the properties mentioned in
Theorem~\ref{notcump}. Namely, if $X$ is a Borel
subset of a complete separable metric space and
$K$ is a bounded Borel subset of $\rbb$ such
that $X$ and $K$ have the same cardinality, then
by \cite[Theorem~2.12]{Par67} there exists a
bijection $\omega_0\colon X \to K$ such that
$\omega_0$ and $\omega_0^{-1}$ are Borel
measurable. This implies that the function
$\omega\colon X \to \mathbb{R}$ defined by
$\omega(x)=\omega_0(x)$ for $x\in X$ is bounded
and $\sigma$-surjective. It turns out that the
notions of injectivity and $\sigma$-surjectivity
coincide for continuous maps $f\colon X \to Y$
between topological Hausdorff spaces whenever
$X$ is $\sigma$-compact (see
\cite[Proposition~16]{C-S-s19}). Coming back to
the case of $X=Y=\rbb$, let us recall the
well-known example of a bounded continuous and
injective (consequently, $\sigma$-surjective)
function $\omega \colon \rbb \to \rbb$ given by
$\omega(x) = \frac{x}{1+|x|}$ for~$x \in \rbb$.
   \bibliographystyle{amsalpha}
   
   \end{document}